\documentclass[a4paper, 11pt, underruledhead, final]{paper}

\usepackage{uwbgeom, enumerate}
\usepackage[mathscr]{euscript}

\usepackage[usenames,dvipsnames]{color}

\def\LineOn(#1,#2){\overline{{#1},{#2}\rule{0em}{1,5ex}}}

\def\luskwiaty{flappy}

\def\lines{{\cal L}}
\def\peki{{\mathcal{P}}}

\def\AffineSpSymb{\mathbf{A}}
\def\AfSpace(#1){\ensuremath{\AffineSpSymb(#1)}}

\def\PencSpace(#1,#2){{\bf P}_{#2}({#1})}

\def\fixproj{\ensuremath{\goth P}}
\def\fixout{\ensuremath{\goth D}}
\def\fixoutf(#1,#2){\mathord{\fixout_{#1}(#2)}}

\def\inc{\mathrel{\,\rule[-3pt]{1pt}{12pt}\,}}
\def\linesout{{\Lines_{\W}}}
\def\linesaff{{\Lines^*}}
\def\pointsout{{S_{\W}}}
\def\indf(#1,#2){\ind_{#1}(#2)}

\def\bundle(#1,#2){\peki_{#1}(#2)}

\let\A\M

\def\W{{\cal W}}

\def\pclique{{\mathcal S}}
\def\stardir(#1){\pclique_\parallel(#1)}
\def\stardirin(#1,#2){\pclique_\parallel^{#2}(#1)}

\def\paral{\mathrel{\parallel_{\W}}}

\newenvironment{cmath}{%
  \par
  \smallskip
  \centering
  $
}{%
  $
  \par
  \smallskip
  \csname @endpetrue\endcsname
}

\newenvironment{ctext}{%
  \par
  \smallskip
  \centering
}{%
 \par
 \smallskip
 \csname @endpetrue\endcsname
}

\begin{document}

\title{The complement of a subspace in a classical polar space}

\author{Krzysztof Petelczyc and Mariusz \.Zynel}

%% \supportauthor{K. Pra\.zmowski}

\maketitle

\begin{abstract}
  In a polar space, embeddable into a projective space, we fix a subspace, that is 
  contained in some hyperplane. The complement of that subspace resembles a slit space
  or a semiaffine space. We prove that under some assumptions the ambient polar space
  can  be recovered in this complement.
\end{abstract}

\begin{flushleft}\small
  Mathematics Subject Classification (2010): 51A15, 51A45.\\
  Keywords: polar space, projective space, semiaffine space, slit space, complement.
\end{flushleft}

%% Section %%%%%%%%%%%%%%%%%%%%%%%%%%%%%%%%%%%%%%%%%%%%%%%%%%%%%%%%%%%%%%%%%%%%%

\section{Introduction}
 
Cohen and Shult coined the term \emph{affine polar space} in \cite{cohenshult} as a polar space
with some hyperplane removed. They prove that from such an affine reduct the ambient polar space 
can be recovered.
In \cite{wielohol} we prove something similar for the complement of a subset in a projective space.
Looking at the results of these two papers one sees that an interesting case has been set aside: the complement of a subspace 
in a polar space. We are trying to fill this this gap here, although under several specific assumptions: we consider 
classical polar spaces, i.e.\ embeddable into projective spaces (cf. \cite{cameron}), 
and our subspace is contained in a hyperplane. 

A projective space with some subspace removed is called a slit space 
(cf. \cite{KM67}, \cite{KP70}, \cite{pianta}) so, our complement can be seen as a generalized slit space.
Singular subspaces in a polar space are projective spaces, in an affine polar space they are affine
spaces (cf. \cite{cohenshult}), while in our complement they are semiaffine or projective spaces.
Adopting the terminology of \cite{Kreuzer}, where the class of semiaffine spaces includes
affine spaces, projective spaces and everything in between, we could say that singular
subspaces of our complement are simply semiaffine spaces. This let us call our complement
a \emph{semiaffine polar space}. 
Anyway, it is clear that the complement we examine is affine in spirit. A natural
parallelism is there and the subspace we remove can be viewed as the horizon.

As this paper is closely related to \cite{cohenshult} and \cite{wielohol}, it borrows some concepts, 
notations and reasonings from these two
works.
There are however new difficulties in this case.
The horizon induces a partial parallelism (cf. \cite{pianta}).
We express this parallelism purely in terms of incidence in the complement.
Then, roughly speaking, the points of the horizon are identified with equivalence 
classes of parallelism, or with directions of lines in other words.
On the horizon of an affine polar space a deep point emerges as the point which
could be reached by no line of the complement.
If the removed subspace is not a hyperplane then there is no deep point but a new problem involving lines arises.
Some lines on the horizon are recoverable in a standard way, as directions of planes.
For the others there are no planes in the complement that would reach them.
An analogy to a deep point is clear, so we call them deep lines. 
To overcome the problem we introduce the following relation:
a line $K$ is anti-euclidean to a line $L$ iff there is no line intersecting $K$ that is parallel to $L$. 
Based on this relation is a ternary collinearity of points on deep lines.

We do not know whether every subspace of a polar space is contained in a hyperplane. Any subspace can be extended to a maximal one, but does it have to be a hyperplane? If that is the case our assumptions could be weakened significantly.

\section{Generalities}

A point-line structure $\A=\struct{S, \lines}$, where the elements of $S$ are
called \emph{points}, the elements of $\lines$ are called \emph{lines}, and
where $\lines\subset2^S$, is said to be a \emph{partial linear space}, or 
a \emph{point-line space}, if two
distinct lines share at most one point and every line is of size (cardinality) at least 2 (cf. \cite{cohen}).
A line of size 3 or more will be called \emph{thick}. If all lines in $\A$ are thick then $\A$ is thick. $\A$ is said to be \emph{nondegenerate} if no point is collinear with all others, 
and it is called \emph{singular} if any two of its points are collinear.
It is called \emph{Veblenian}  iff 
for any two distinct lines $L_1, L_2$ through a point $p$ and
any two distinct lines $K_1, K_2$ not through the point $p$
whenever each of $L_1,L_2$ intersects both of $K_1, K_2$, then $K_1$ intersects $K_2$. A \emph{subspace} of $\A$ is a subset $X\subseteq S$ that contains every line, which meets $X$ in at least two points.
A proper subspace of $\A$ that shares a point with every line is said to be a \emph{hyperplane}. If $\A$ satisfies exchange axiom, then a \emph{plane} of $\A$ is a singular subspace of dimension $2$.

A partial linear space satisfying \emph{one-or-all} axiom, that is
\begin{ctext}
for every $L\in\lines$ and $a\notin L$, $a$ is collinear with one or all points on $L$, 
\end{ctext}
will be called a \emph{polar space}.
The rank of a polar space is the maximal number $n$ for which there is a chain of singular subspaces $\emptyset\neq X_1\subset X_2\subset\ldots\subset X_{n}$ ($n=-1$ if this chain is reduced to the empty set).
For $a\in S$ by $a^{\perp}$ we denote the set of all points collinear with $a$,
and for $X\subseteq S$ we put
 $$X^{\perp}=\bigcap\{a^{\perp}\colon a\in X\},\quad \rad X = X\cap X^\perp.$$
As an immediate consequence of one-or-all axiom we get (cf. \cite{cohenshult}):

\begin{fact}\label{fact:aorto-hipa}
  For any point $a\in S$ the set $a^{\perp}$ is a hyperplane of\/ $\fixproj$ .
\end{fact}

Following \cite{segr2afin}, a subset $X$ of $S$ is called 
\begin{itemize}\itemsep-2pt
  \item
  \emph{spiky} when every point $a\in X$ is collinear with 
  some point $b\notin X$,
  \item
  \emph{\luskwiaty} when for every line $L\subseteq X$ there is a point
  $a\notin X$ such that $L\subseteq a^{\perp}$.
\end{itemize}

%%%%%%%%%%%%%%%%%%%%%%%%%%%%%%%%%%%%%%%%%%%%% 

\subsection{Complement}
Let $\A=\struct{S,\lines}$ be a  thick partial linear space 
and let $\W$ be its proper subspace.
By the \emph{complement of\/ $\W$ in $\A$} we mean the structure
  $$\fixoutf(\A,\W) := \struct{\pointsout, \linesout},$$
where
  $$\pointsout := S\setminus\W \qquad\text{and}\qquad
     \linesout := \{ k\cap\pointsout\colon k\in\lines \Land k\nsubseteq\W \}.$$
The subspace  $\W$ will be called the \emph{horizon of\/ $\fixoutf(\A,\W)$}.
Note that the complement $\fixoutf(\A,\W)$ is a partial linear space.
Following a standard convention we call the points and lines of the
complement $\fixoutf(\A,\W)$ \emph{proper}, and points and lines
of $\W$ are said to be 
\emph{improper}.
By the \emph{closure} of a proper line $L$ we mean the line $\overline{L}\in\lines$
with $L\subseteq\overline{L}$.
Similarly we will use and denote closure of any subspace of  $\fixoutf(\A,\W)$.

We say that two lines $K,L\in\linesout$  are \emph{parallel}, and we write 
\begin{equation}\label{eq:comp-paral}
K\paral L\quad\text{iff}\quad \overline{K}\cap \overline{L}\cap\W\neq\emptyset.
\end{equation}
Note that $\paral$ is an equivalence relation.
A line $L\in\linesout$ with the property that $L\paral L$ will be called an \emph{affine line}. 
The set of all affine lines will be denoted by $\linesaff$.
For affine line $L$ we write
$L^\infty$ for the point of $\overline{L}$ in $\W$, i.e.\ the point at infinity.
A point $a\in\W$ is said to be a \emph{deep point} if there is no line  $L\in\linesout$ such that $a=L^\infty$.
 A plane of $\fixoutf(\A,\W)$ containing an affine line is said to be a \emph{semiaffine plane}. 
 By $\Pi^{\infty}$ we denote the set of points at infinity of semiaffine plane $\Pi$, i.e.\
$\Pi^{\infty}=\{M^{\infty}\colon M\in\linesaff \text{ and } M\subseteq\Pi\}$. 
A line $L\subseteq\W$ is said to be a \emph{deep line} if there is no plane 
in $\fixoutf(\fixproj,\W)$ with $L=\Pi^{\infty}$.

%%%%%%%%%%%%%%%%%%%%%%%%%%%%%%%%%%%%%%%%%

\section{Complement in a polar space}

Let $\fixproj=\struct{S,\lines}$ be a thick, nondegenerate polar space of rank at least 3.
Assume that $\W$ is a proper subspace of $\fixproj$, that is contained in some hyperplane.
We deal with the complement  $\fixoutf(\fixproj,\W)$.

We can determine the number of deep points in hyperplanes of polar spaces. It turns out, that deep points appear only on hyperplanes.

\begin{lem}\label{lem:deeppoints}
 %Let $\fixproj$ be a nondegenerate, thick polar space of rank at least three.
\begin{sentences}
\item\label{fact:deephyper}
  If\/ $\W$ is a hyperplane in \/ $\fixproj$, then there is at most one deep point in $\W$ and it is in $\rad\W$.
\item\label{lem:sub-spiky}
  If\/ $\W$ is not a hyperplane in \/ $\fixproj$, then there are no deep points in $\W$, that is $\W$ is spiky.
\end{sentences}
\end{lem}
\begin{proof}
(i): By Corollary 1.3 (ii) in \cite{cohenshult}.

(ii): Assume that $a$ is a deep point of $\W$. Then $a^{\perp}\subseteq\W$, and by \ref{fact:aorto-hipa}
we get that $\W$ contains a hyperplane. It yields a contradiction, as hyperplane in $\fixproj$ is a maximal proper subspace (cf.  \cite[1.1]{cohenshult}).
\end{proof}

\begin{lem}\label{lem:extend-no-lines}
  Let\/  $\fixproj$ be embeddable polar space and\/ $K,L\in\linesout$ be two distinct 
  lines such that $K\paral L$.
  The subspace $\W$ can be extended to a hyperplane of\/ $\fixproj$ not containing 
  $\overline{K}$ and $\overline{L}$.
\end{lem}
\begin{proof}
If $\W$ is a hyperplane of $\fixproj$ then $\W$ itself is the required hyperplane. 

Assume that $\W$ is not a hyperplane.
Let $H$ be a hyperplane containing $\W$, $\goth N$ be a projective space embracing $\fixproj$, and $f$ be an embedding of $\fixproj$ into $\goth N$.
Consider the projective subspace $G$ spanned by $f(H)$. By \cite{cohenshult} $G$ is a hyperplane of $\goth N$.  If $f(\overline{K}),f(\overline{L})\nsubseteq G$ then our hyperplane $H=f^{-1}(G\cap f(S))$ is the required one. 

Assume that $f(\overline{K})\subseteq G$ or $f(\overline{L})\subseteq G$. In case $f(\W)$ is not a hyperplane in $G$, consider a family $\cal H$ of hyperplanes in $G$ containing $f(\W)$. For $f(\overline{K})\subseteq G$ and $f(\overline{L})\subseteq G$ we
take $a_K\in f(\overline{K})\setminus f(\W)$, $a_L\in f(\overline{L})\setminus f(\W)$ and choose a hyperplane $G_0\in \cal H$ with $a_K,a_L\notin G_0$. If
$f(\overline{K})\nsubseteq G$ or $f(\overline{L})\nsubseteq G$  one of the points $a_K$, $a_L$ is enough and then we set $G_0\in \cal H$ with $a_K\notin G_0$ or $a_L\notin G_0$, respectively.
For $i=K,L$, if $\overline{a_i,b}\cap G_0\neq\emptyset$ then $\overline{a_i,b}\subseteq G$, that contradicts $b\notin G$. So, $\overline{a_i,b}\cap G_0=\emptyset$.
Then, $\struct{G_0,b}=G'$ is a hyperplane of $\goth N$. Moreover, $f(\W)\subseteq G'$ and $f(\overline{K})\nsubseteq G'$, $f(\overline{L})\nsubseteq G'$. Thus, $H':=f^{-1}(G'\cap f(S))$ is the hyperplane we are looking for. 
\end{proof}

\begin{lem}\label{lem:connected}
  Let\/ $K,L\in\linesout$ be two distinct lines such that $K\paral L$.
  There is a sequence $\Pi_1,\ldots,\Pi_n$ of planes in   
  $\fixoutf(\fixproj,\W)$ such that $K^\infty=L^\infty\in\overline{\Pi_i}$ for $i=1,\ldots,n$ and
  $K\subseteq\Pi_1$, $L\subseteq \Pi_n$, and 
  $\Pi_j, \Pi_{j+1}$ share a line for $j=1,\ldots,n-1$.
\end{lem}
\begin{proof}
  By \ref{lem:extend-no-lines} we can extend $\W$ to a hyperplane $H$ of
  $\fixproj$  such that $K,L\nsubseteq H$. Take the point $a=K^\infty$. By
  \eqref{eq:comp-paral} we have $a=L^\infty$.   Now, take in $\fixproj$ the
  bundle of all the lines together with all the  planes through $a$.  This
  structure is, up to an isomorphism, a polar space $\fixproj'$, so called
  quotient polar space (cf. \cite{cohenshult}).
  The set $H'$, consisting of all the lines through $a$ contained in $H$, is a hyperplane in $\fixproj'$ induced by $H$.
  Then $\fixoutf(\fixproj',H')$ is an affine polar space, that in itself is
  connected (cf. \cite{cohenshult}). So there is in  $\fixoutf(\fixproj',H')$ a
  sequence of intersecting lines joining $K$ and $L$ as points of
  $\fixoutf(\fixproj',H')$. However, lines of $\fixoutf(\fixproj',H')$ are
  planes of $\fixoutf(\fixproj,H)$. As $\W\subseteq H$ these planes are also
  planes of $\fixoutf(\fixproj,\W)$.
\end{proof}

\subsection{Parallelism}

Let $K_1, K_2\in\linesout$.
Then 
\begin{multline}\label{eq:Vebl-paral}
  K_1 \parallel^\ast K_2\quad\text{iff}\quad
   K_1\cap K_2=\emptyset \text{ and there are two} 
      \text{ distinct lines } L_1, L_2\in\linesout \\
        \text{ crossing both of } K_1, K_2, \text{ such that } L_1\cap L_2\neq\emptyset.
\end{multline}
Let $\parallel$ be the transitive closure of $\parallel^\ast$.
It is clearly seen that $\parallel\subseteq\linesaff\times\linesaff$.

\begin{lem}
  The relation $\parallel$ is reflexive on $\linesaff$.
\end{lem}
\begin{proof}
  Given a line $K_1\in\linesout$, considering that the rank of $\fixproj$ is at least 3, 
  take a plane $\pi$ containing $K_1$ in a maximal singular subspace through $K_1$. 
  There are lines 
  $K_2, L_1, L_2$ on $\pi$ such that $K_1\cap K_2=\emptyset$ (that is $K_1^\infty = K_2^\infty$), 
  $L_1\neq L_2$, $L_1\cap L_2\neq\emptyset$, and 
  $K_i\cap L_j\neq\emptyset$ for $i,j=1,2$. Thus $K_1\parallel^\ast K_2$ by \eqref{eq:Vebl-paral}.
  This means that $K_1\parallel K_2$ and $K_2\parallel K_1$, which by transitivity implies that  
  $K_1\parallel K_1$.
\end{proof}

\begin{prop}\label{prop:paral-coinc}
  Let $\W$ be a subspace of\/ $\fixproj$.
  The relation $\paral$ defined in \eqref{eq:comp-paral} 
  and the relation $\parallel$ coincide on the set of lines of\/ $\fixoutf(\fixproj,\W)$.
\end{prop}
\begin{proof}
  Let $K_1, K_2\in\linesout$.
  If $K_1=K_2$, then $K_1\paral K_2$ and $K_1\parallel K_2$. So, assume that $K_1\neq K_2$.

  Consider the case where $K_1\paral K_2$. By \eqref{eq:comp-paral} it means that
  $\overline{K}\cap \overline{L}\cap W\neq\emptyset$, and consequently $K_1^\infty=K_2^\infty=a$ for some $a\in\W$. 
  This implies that $K_1\cap K_2=\emptyset$. Assume that $K_1$ and $K_2$ are coplanar, and $\Pi$ is the plane of $\fixoutf(\fixproj,\W)$ containing both of $K_1,K_2$. The plane $\overline{\Pi}$ is, up to an isomorphism, a projective plane, so it is Veblenian. Thus, by \eqref{eq:Vebl-paral}, $K_1\parallel^\ast K_2$. If $K_1$ and $K_2$ are not coplanar, then by  
  \ref{lem:connected} there is a sequence of planes $\Pi_1,\ldots,\Pi_n$  such that $K_1\subseteq\Pi_1$, $K_2\subseteq \Pi_n$,  $a\in\overline{\Pi_i}$ for $i=1,\ldots,n$, and
  $\Pi_j, \Pi_{j+1}$ share a line for $j=1,\ldots,n-1$. Let $\Pi_j\cap \Pi_{j+1}=M_j$. Note that $a\in \overline{M_1},\ldots,\overline{M_{n-1}}$ and $M_j, M_{j+1}$ are coplanar. Therefore $M_j\parallel^\ast M_{j+1}$. Moreover, $K_1\parallel^\ast M_1$ and $M_{n-1}\parallel^\ast K_2$ by the same reasons. So finally we get
  $K_1\parallel K_2$.

  Now, assume that $K_1\parallel^\ast K_2$. 
  Then $K_1, K_2$ are disjoint and coplanar.
  Thus $\overline{K_1}, \overline{K_2}$ meet in the closure of some plane, this means that they meet in $\W$. By \eqref{eq:comp-paral} it gives
  $K_1\paral K_2$. If  $K_1\parallel K_2$ then there is a sequence of proper lines
  $L_1,\ldots, L_n$ such that $K_1\parallel^\ast L_1\parallel^\ast\ldots\parallel^\ast L_n\parallel^\ast K_2$. 
  So, from the previous reasoning we get $K_1\paral L_1\paral\ldots\paral L_n\paral K_2$. 
  As the relation $\paral$ is transitive we have $K_1\paral K_2$.
\end{proof}

%%%%%%%%%%%%%%%%%%%%%%%%%%%%%%%%%%%%%%%%%%%%%%%%%

As an immediate consequence of \ref{prop:paral-coinc} we get

\begin{cor}\label{cor:affine-lines}
  Affine lines can be distinguished in the set $\linesout$ as those parallel to themselves.
\end{cor}

%%%%%%%%%%%%%%%%%%%%%%%%%%%%%%%%%%%%%%

\subsection{Recovering}

%If $\W$ is a hyperplane the result follows by 2.7 of \cite{cohenshult}, namely:

If $\W$ is a hyperplane it follows by \cite[2.7]{cohenshult} that:

\begin{prop}\label{prop:coh-shult}
  Let\/ $\fixproj$ be a thick nondegenerate polar space of rank at least 2
  and let\/ $H$ be its hyperplane. 
  The polar space $\fixproj$ can be recovered in the complement\/ $\fixoutf(\fixproj,H)$.
\end{prop}
\noindent
So, from now on we additionally assume that $\W$ is not a hyperplane.

By \ref{prop:paral-coinc} the relation $\paral$, which is the natural parallelism in our complement $\fixoutf(\fixproj,\W)$, can be expressed purely in terms of $\fixoutf(\fixproj,\W)$. Note that our parallelism is partial: it is defined only on affine lines. However it is not a problem in view of \ref{cor:affine-lines}.
From \ref{lem:deeppoints}\eqref{lem:sub-spiky} there is a bijection between the sets
$\W=\{L^\infty\colon L\in \linesaff\}$ and $\{[L]_{\parallel}\colon L\in\linesaff\}$. Thus we can recover $\W$ pointwise in a standard way:
\begin{ctext}
 points of the horizon $\W$ are identified with equivalence classes of parallelism \\ i.e. directions of affine lines of the complement $\fixoutf(\fixproj,\W)$. 
\end{ctext}

Let us introduce a relation $\mathord{\sim}\subseteq\linesaff\times\linesaff$ 
defined by the following condition:
\begin{equation}\label{eq:rel-falka}
  K_1\sim K_2 \;\iff \;\pforall{a\in K_1} \pforall{M\in\linesaff} [\;a\in M\Rightarrow M\nparallel K_2\;].
\end{equation}
In the sense of Euclid's Fifth Postulate it could be read as \emph{anti-euclidean parallelism}.
A lot more useful for us is its derivative 
%The relation $\sim$ can be considered as "anti-euclidean" in the sense of the famous Euclid's Fifth Postulate. 
%We use this relation to define another relation
$\mathord{\equiv}\subseteq\linesaff/_\parallel\times\linesaff/_\parallel$ defined as follows:
\begin{equation}\label{eq:rel-kreski}
[K_1]_\parallel \; \equiv \; [K_2]_{\parallel}\;\iff\; 
\pforall{M\in [K_1]_\parallel}\pforall{N\in [K_2]_\parallel}[\;M\sim N \text{ and }N\sim M\;].
\end{equation}

\begin{lem}\label{lem:equiv}
  Let\/ $M$, $N$ be two nonparallel affine lines. The following conditions are equivalent:
  \begin{sentences}
  \item
    $[M]_\parallel\equiv [N]_\parallel$,
  \item
    there is a deep line $L\subseteq\W$, such that $M^\infty,N^\infty\in L$.
  \end{sentences}
\end{lem}
\begin{proof}
  (i)$\;\Rightarrow\;$(ii):
  From one-or-all axiom, $M^\infty$ must be collinear with at least one point of the 
  line $\overline{N}$. Moreover, $M^\infty$ cannot be collinear with a proper point 
  of $\overline{N}$, as  $[M]_\parallel\equiv [N]_\parallel$. Thus $M^\infty$ is 
  collinear with the unique improper point of $\overline{N}$, which is $N^\infty$. 

  Let $L$ be the line through $M^\infty$, $N^\infty$. Assume, that $\Pi$ is a 
  semiaffine plane with $L={\Pi}^\infty$. Then, there are some affine lines 
  $M_1, N_1\subseteq\Pi$  with $M^\infty=M_1^\infty$ and $N^\infty=N_1^\infty$. 
  So, either $M_1\parallel N_1$ or $M_1$ and $N_1$ share a proper point. 
  In view of \eqref{eq:rel-kreski}, in both cases we get $[M]_\parallel\not\equiv [N]_\parallel$.

  (ii)$\;\Rightarrow\;$(i):
  Assume that $[M]_\parallel\not\equiv [N]_\parallel$. Due to \eqref{eq:rel-falka} 
  and \eqref{eq:rel-kreski} there is a proper point $a\in M$ and an affine line $K$ 
  such that $a\in K\parallel N$ (or the symmetrical case holds). This means that $a$ 
  and $N^\infty$ are collinear in $\fixproj$.
  The one-or-all axiom implies, that either there are no other points on $M$ that
   are collinear with $N^\infty$, or $N^\infty$ is collinear with all points on $M$. 
   In the first case $N^\infty$ is not collinear with $M^\infty$, in the latter 
   $\gen{N^\infty,M}\nsubseteq\W$ is the plane containing the line $\overline{M^\infty, N^\infty}$.
\end{proof}

One can note, that the relation $\equiv$ defined by \eqref{eq:rel-kreski} and 
the relation $\equiv$ introduced in \cite{cohenshult} coincide, though their 
definitions are expressed differently. Besides, our relation is not transitive, 
but the reflexive closure of its analogue in \cite{cohenshult} is an equivalence 
relation. This benefit is strictly caused by some hyperplane properties 
(see \ref{lem:deeppoints}\eqref{fact:deephyper}). Nevertheless, we can overcome 
this inconvenience and  define ternary relation of collinearity on the horizon $\W$.

\begin{lem}\label{lem:strange-lines}
  If\/ $K_1$, $K_2$, $K_3$ are pairwise nonparallel affine lines such that
  $[K_i]_\parallel\equiv [K_{i+1\mod 3}]_\parallel$ for $i=1,2,3$, 
  then points $K_1^\infty$, $K_2^\infty$, $K_3^\infty$ are on a line.
\end{lem}
\begin{proof}
  Let $a=K_1^\infty$, $b=K_2^\infty$, $c=K_3^\infty$.
  By \ref{lem:equiv} there are improper lines $L=\overline{a,b}$, $M=\overline{b,c}$, 
  $N=\overline{c,a}$. Let $H$ be a hyperplane containing $\W$. 
  If in $\fixoutf(\fixproj,H)$ there is a plane, which closure contains one of 
  the lines $L$, $M$ or $N$, then we also have such plane in  $\fixoutf(\fixproj,\W)$, 
  that contradicts \ref{lem:equiv}. Thus, $L,M,N\subseteq H$ are deep lines in relation 
  to $\fixoutf(\fixproj,H)$. By 2.3 of \cite{cohenshult} this means that each of $L$, $M$ and $N$ 
  contains a point of $\rad H$. Let $d\in\rad H$.  Then, by  1.3 of \cite{cohenshult}, 
  $H=d^\perp$, $\{d\}=\rad H$, and $d$ is the unique deep point of $H$.
  As we have $d\in L,M,N$, it must be $L=M=N$.
\end{proof}

\begin{lem}\label{lem:ternary-collin}
  Let $K_1$, $K_2$, $K_3$ be pairwise nonparallel affine lines.
  Points $K_1^\infty$, $K_2^\infty$, $K_3^\infty$
  are on a line iff one of the following holds:
  \begin{sentences}
  \item  
    there are affine lines $M_1\parallel K_1$, $M_2\parallel K_2$,
    $M_3\parallel K_3$ such that  $M_1, M_2, M_3$  form a triangle in $\fixoutf(\fixproj,\W)$, 
  \item
    $[K_1]_\parallel\equiv [K_2]_\parallel$, $[K_2]_\parallel\equiv [K_3]_\parallel$, 
    and $[K_3]_\parallel\equiv [K_1]_\parallel$.
  \end{sentences}
\end{lem}
\begin{proof}
  Assume that $K_1^\infty$, $K_2^\infty$, $K_3^\infty$ are on a line $L$. If (i)
  does not hold, then there is no plane $\Pi$ in $\fixoutf(\fixproj,\W)$ with
  $L=\Pi^{\infty}$. This means that $L$ is a deep line and by \ref{lem:equiv} we
  get (ii). 

  Now, assume that (i) is the case. Take a plane $\Pi$ spanned by the triangle $M_1,
  M_2, M_3$.  Then $K_1,K_2,K_3\subseteq \Pi$ and $K_1^\infty$, $K_2^\infty$,
  $K_3^\infty$ are on a  line $\Pi^\infty$. If (ii) is fulfilled then
  $K_1^\infty$, $K_2^\infty$, $K_3^\infty$ are on a line directly by
  \ref{lem:strange-lines}.
\end{proof}

The meaning of \ref{lem:ternary-collin} is that we are able to recover improper lines 
regardless of whether $\W$ is flappy or not.
Let $\bigl[[K]_\parallel, [L]_\parallel \bigr]_{\equiv}:=
             \bigl\{[M]_\parallel\colon [M]_\parallel\equiv [K]_\parallel,
[L]_\parallel\bigr\}$.
Then new lines can be grouped into two sets:
\begin{multline*}
 \lines' := \Bigl\{\bigl[[K]_\parallel, [L]_\parallel \bigr]_{\equiv}\colon
 [K]_\parallel\equiv [L]_\parallel \text{ and } K \nparallel L \Bigr\}, \\
  \lines'' := \bigl\{\Pi^\infty\colon \Pi \text{ is a semiaffine plane of } \fixoutf(\fixproj,\W)\bigr\}.
\end{multline*}
All our efforts in this paper essentially amount to the following isomorphism 
\begin{cmath}
  \fixproj\cong\bstruct{\pointsout\cup\linesaff/_\parallel,\; \linesout\cup\lines'\cup\lines'',\; \mathord{\inc}}.
\end{cmath}
A new point $[K]_\parallel$ is incident to a line $L\in\linesout$ iff $K\parallel L$.
It is incident to a line $L\in\lines'$ iff there is $M\in\linesout$ such that $\big[[K]_\parallel, [M]_\parallel\big]_{\equiv}= L$. 
Eventually, it is incident to a line $L\in\lines''$ iff $K\subseteq\Pi$ and $L=\Pi^\infty$.

\iffalse 
The incidence relation $\inc$ between "new" points and lines is defined by
\begin{itemize}\itemsep-2pt
\item
$[K]_\parallel\inc L$ iff $K\parallel L$, for $L\in\linesout$,
\item
$[K]_\parallel\inc L$ iff $\big[[K]_\parallel\big]_{\equiv^\ast}= L$, for $L\in\lines'$,
\item
$[K]_\parallel\inc L$ iff $K\subseteq\Pi$ and $L=\Pi^\infty$, for $L\in\lines''$.
\end{itemize}
Gathering all assumptions, finally we obtain the following result:
\fi

\begin{thm}\label{thm:the-main}
  Let\/ $\fixproj$ be a thick, nondegenerate, embeddable polar space of rank at least 3, 
  and $\W$ be its subspace, that is contained in a hyperplane.
 The polar space $\fixproj$ can be recovered in the complement\/ $\fixoutf(\fixproj,\W)$.
\end{thm}

%%%%%%%%%%%%%%%

%% Contact %%%%%%%%%%%%%%%%%%%%%%%%%%%%%%%%%%%%%%%%%%%%%%%%%%%%%%%%%%%%%%%%%%%%%

\begin{flushleft}
  \noindent\small
  K. Petelczyc, M. \.Zynel\\
  Institute of Mathematics, University of Bia{\l}ystok,\\
  K. Cio{\l}kowskiego 1M, 15-245 Bia{\l}ystok, Poland\\
  \verb+kryzpet@math.uwb.edu.pl+,
  \verb+mariusz@math.uwb.edu.pl+
\end{flushleft}

\end{document}